\documentclass[11pt,fleqn]{article}

\usepackage{amsmath,amssymb,amsthm}
\usepackage[margin=1in]{geometry}
\usepackage[pdfpagemode=UseNone,pdfstartview=FitH]{hyperref}

\newcommand{\abs}[1]{\left|#1\right|}
\newcommand{\bdry}[1]{\partial #1}
\newcommand{\bgset}[1]{\big\{#1\big\}}

\newcommand{\dualp}[2]{#1\, #2}

\newcommand{\incl}{\subset}
\newcommand{\isom}{\approx}
\newcommand{\norm}[2][]{\left\|#2\right\|_{#1}}
\renewcommand{\o}{\text{o}}
\newcommand{\PS}[1]{$(\text{PS})_{#1}$}
\newcommand{\pnorm}[2][]{\if #1'' \left|#2\right|_p \else \left|#2\right|_{#1} \fi}
\newcommand{\QED}{\mbox{\qedhere}}

\newcommand{\seq}[1]{\left(#1\right)}
\newcommand{\set}[1]{\left\{#1\right\}}

\newcommand{\A}{{\cal A}}
\newcommand{\F}{{\cal F}}
\newcommand{\M}{{\cal M}}
\newcommand{\R}{\mathbb R}
\newcommand{\RP}{\R \text{P}}
\newcommand{\Z}{\mathbb Z}

\DeclareMathOperator{\divg}{div}
\DeclareMathOperator{\essinf}{ess\, inf}

\newenvironment{enumarab}{\begin{enumerate}

}{\end{enumerate}}

\newenvironment{enumroman}{\begin{enumerate}

}{\end{enumerate}}

\newtheorem{lemma}{Lemma}[section]
\newtheorem{proposition}[lemma]{Proposition}
\newtheorem{theorem}[lemma]{Theorem}

\theoremstyle{definition}
\newtheorem{definition}[lemma]{Definition}

\numberwithin{equation}{section}

\title{\bf $p$-Laplace equations with singular weights\thanks{ \smallskip}}
\author{\bf Kanishka Perera\thanks{This work was completed while the first-named author was visiting the Department of Mathematics at the University of Ulsan, and he is grateful for the kind hospitality of the department.}\\
Department of Mathematical Sciences\\
Florida Institute of Technology\\
Melbourne, FL 32901, USA\\
\em kperera@fit.edu\\
[\bigskipamount]
\bf Inbo Sim\thanks{This work was supported by the 2012-0117 Research Fund of the University of Ulsan.}\\
Department of Mathematics\\
University of Ulsan\\
Ulsan 680-749, Republic of Korea\\
\em ibsim@ulsan.ac.kr}
\date{}

\begin{document}

\maketitle

\begin{abstract}
We study a class of $p$-Laplacian Dirichlet problems with weights that are possibly singular on the boundary of the domain, and obtain nontrivial solutions using Morse theory. In the absence of a direct sum decomposition, we use a cohomological local splitting to get an estimate of the critical groups.
\end{abstract}

\newpage

\section{Introduction}

Let $\Omega$ be a bounded domain in $\R^N$, $N \ge 2$, with Lipschitz boundary $\bdry{\Omega}$. The purpose of this paper is to study the boundary value problem
\begin{equation} \label{1.1}
\left\{\begin{aligned}
- \Delta_p\, u & = f(x,u) && \text{in } \Omega\\[10pt]
u & = 0 && \text{on } \bdry{\Omega},
\end{aligned}\right.
\end{equation}
where $\Delta_p\, u = \divg (|\nabla u|^{p-2}\, \nabla u)$ is the $p$-Laplacian of $u$, $p \in (1,\infty)$, and $f$ is a Carath\'{e}odory function on $\Omega \times \R$ satisfying a subcritical growth condition of the form
\begin{equation} \label{1.2}
|f(x,t)| \le \sum_{i=1}^n K_i(x)\, |t|^{q_i - 1} \quad \text{for a.a. $x \in \Omega$ and all $t \in \R$}
\end{equation}
for some $q_i \in [1,p^\ast)$, $p^\ast = Np/(N-p)$ if $p < N$ and $p^\ast = \infty$ if $p \ge N$, and measurable weights $K_i$ that are possibly singular on $\bdry{\Omega}$.

Elliptic boundary value problems with singular weights have become an increasingly active area of research during recent years. Cuesta \cite{MR1836801} studied the eigenvalue problem
\begin{equation} \label{1.3}
\left\{\begin{aligned}
- \Delta_p\, u & = \lambda\, K(x)\, |u|^{p-2}\, u && \text{in } \Omega\\[10pt]
u & = 0 && \text{on } \bdry{\Omega},
\end{aligned}\right.
\end{equation}
where the weight $K \in L^s(\Omega)$ with $s > N/p$ if $p \le N$ and $s = 1$ if $p > N$. In the ODE case $N = 1$, Kajikiya, Lee, and Sim \cite{MR2409516} considered weights that are strongly singular on the boundary and may not be in $L^1$. Montenegro and Lorca \cite{MR2914567} used the Hardy-Sobolev inequality to study the spectrum of \eqref{1.3} for a wide class of singular weights $\A_p$ (see Definition \ref{Definition 2.1}). In the semilinear case $p = 2$, Kajikiya \cite{MR2674189} showed that problem \eqref{1.1} with $f(x,u) = K(x)\, |u|^{q-1}\, u$, $q > 1$, and $K \in \A_q$ has a positive solution and infinitely many solutions without positivity using variational methods. Kajikiya, Lee, and Sim \cite{MR2833356} obtained positive and nodal solutions of \eqref{1.1} when $N = 1$ and $f$ is strongly singular on the boundary and $p$-superlinear at infinity using bifurcation arguments.

Here we study the critical groups of the variational functional associated with problem \eqref{1.1} and obtain nontrivial solutions using Morse theory. The admissible class of weights is defined in Section \ref{Section 2}. After the preliminaries on a related eigenvalue problem in Section \ref{Section 3}, we use a cohomological local splitting to get an estimate of the critical groups in the absence of a direct sum decomposition in Section \ref{Section 4}. Our main existence result is for the $p$-superlinear case and is proved in Section \ref{Section 5} (see Theorem \ref{Theorem 5.6}).

\section{Variational setting} \label{Section 2}

Let $\pnorm[p]{\cdot}$ denote the norm in $L^p(\Omega)$, and let $\rho(x) = \inf_{y \in \bdry{\Omega}}\, |x - y|$ be the distance from $x \in \Omega$ to $\bdry{\Omega}$. We consider the following class of weights.

\begin{definition} \label{Definition 2.1}
For $q \in [1,p^\ast)$, let $\A_q$ denote the class of measurable functions $K$ such that $K \rho^a \in L^r(\Omega)$ for some $a \in [0,q - 1]$ and $r \in (1,\infty)$ satisfying
\begin{equation} \label{2.1}
\frac{1}{r} + \frac{a}{p} + \frac{q-a}{p^\ast} < 1.
\end{equation}
\end{definition}

Let $W^{1,\, p}_0(\Omega)$ be the usual Sobolev space, with the norm $\norm{u} = \pnorm[p]{\nabla u}$, and let $C$ denote a generic positive constant.

\begin{lemma} \label{Lemma 2.2}
If $q \in [1,p^\ast)$ and $K \in \A_q$, then there exists $b < p^\ast$ such that
\[
\int_\Omega |K(x)|\, |u|^{q-1}\, |v|\, dx \le C \norm{u}^{q-1} \pnorm[b]{v} \quad \forall u, v \in W^{1,\, p}_0(\Omega).
\]
\end{lemma}

\begin{proof}
Let $a$ and $r$ be as in Definition \ref{Definition 2.1}. By the H\"older inequality,
\[
\int_\Omega |K(x)|\, |u|^{q-1}\, |v|\, dx = \int_\Omega \abs{K \rho^a} \abs{\frac{u}{\rho}}^a |u|^{q-1-a}\, |v|\, dx \le \pnorm[r]{K \rho^a} \pnorm[p]{\frac{u}{\rho}}^a \pnorm[b]{u}^{q-1-a} \pnorm[b]{v},
\]
where $1/r + a/p + (q-a)/b = 1$ and hence $b < p^\ast$ by \eqref{2.1}. Since $\pnorm[p]{u/\rho} \le C \norm{u}$ by the Hardy inequality (see Ne{\v{c}}as \cite{MR0163054}) and $\pnorm[b]{u} \le C \norm{u}$ by the Sobolev imbedding, the conclusion follows.
\end{proof}

We assume that $K_i \in \A_{q_i}$ for $i = 1,\dots,n$. Recall that a weak solution of problem \eqref{1.1} is a function $u \in W^{1,\, p}_0(\Omega)$ satisfying
\[
\int_\Omega |\nabla u|^{p-2}\, \nabla u \cdot \nabla v\, dx = \int_\Omega f(x,u)\, v\, dx \quad \forall v \in W^{1,\, p}_0(\Omega).
\]
By \eqref{1.2} and Lemma \ref{Lemma 2.2}, the integral on the right is well-defined. Weak solutions coincide with critical points of the $C^1$-functional
\begin{equation} \label{2.2}
\Phi(u) = \int_\Omega \left[\frac{1}{p}\, |\nabla u|^p - F(x,u)\right] dx, \quad u \in W^{1,\, p}_0(\Omega),
\end{equation}
where $F(x,t) = \int_0^t f(x,\tau)\, d\tau$. Lemma \ref{Lemma 2.2} also gives the following compactness result.

\begin{lemma} \label{Lemma 2.3}
Every bounded sequence $\seq{u_j} \subset W^{1,\, p}_0(\Omega)$ such that $\Phi'(u_j) \to 0$ has a convergent subsequence.
\end{lemma}

\begin{proof}
By Lemma \ref{Lemma 2.2},
\begin{equation} \label{2.3}
\abs{\int_\Omega f(x,u_j)\, (u_j - u)\, dx} \le \sum_{i=1}^n \int_\Omega K_i(x)\, |u_j|^{q_i - 1}\, |u_j - u|\, dx \le C \sum_{i=1}^n \norm{u_j}^{q_i - 1} \pnorm[b_i]{u_j - u},
\end{equation}
where each $b_i < p^\ast$. Since $\seq{u_j}$ is bounded in $W^{1,\, p}_0(\Omega)$, a renamed subsequence converges to some $u$ weakly in $W^{1,\, p}_0(\Omega)$ and strongly in $L^{b_i}(\Omega)$ for $i = 1,\dots,n$. Then
\[
\int_\Omega |\nabla u_j|^{p-2}\, \nabla u_j \cdot \nabla (u_j - u)\, dx = \dualp{\Phi'(u_j)}{(u_j - u)} + \int_\Omega f(x,u_j)\, (u_j - u)\, dx \to 0
\]
by \eqref{2.3}, and hence $u_j \to u$ in $W^{1,\, p}_0(\Omega)$ by the (S)$_+$ property of the $p$-Laplacian.
\end{proof}

\section{An eigenvalue problem} \label{Section 3}

In this section we consider the eigenvalue problem
\begin{equation} \label{3.1}
\left\{\begin{aligned}
- \Delta_p\, u & = \lambda\, h(x)\, |u|^{q-2}\, u && \text{in } \Omega\\[10pt]
u & = 0 && \text{on } \bdry{\Omega},
\end{aligned}\right.
\end{equation}
where $q \in [1,p^\ast)$ and $h \in \A_q$ is possibly sign-changing. We assume that $h$ is positive on a set of positive measure and seek positive eigenvalues.

Let
\[
I(u) = \int_\Omega \frac{1}{p}\, |\nabla u|^p\, dx, \quad J(u) = \int_\Omega \frac{1}{q}\, h(x)\, |u|^q\, dx, \quad u \in W^{1,\, p}_0(\Omega),
\]
and set
\[
\Psi(u) = \frac{1}{J(u)}, \quad u \in \M = \bgset{u \in W^{1,\, p}_0(\Omega) : I(u) = 1 \text{ and } J(u) > 0}.
\]
Then $\M$ is nonempty, and positive eigenvalues and associated eigenfunctions of problem \eqref{3.1} on $\M$ coincide with critical values and critical points of $\Psi$, respectively. By Lemma \ref{Lemma 2.2},
\[
0 < J(u) \le C \norm{u}^q \le C \quad \forall u \in \M
\]
and hence $\lambda_1 := \inf_{u \in \M}\, \Psi(u) > 0$.

\begin{lemma}
For all $c \in \R$, $\Psi$ satisfies the {\em \PS{c}} condition, i.e., every sequence $\seq{u_j} \subset \M$ such that $\Psi(u_j) \to c$ and $\Psi'(u_j) \to 0$ has a subsequence that converges to some $u \in \M$.
\end{lemma}

\begin{proof}
We have $c \ge \lambda_1$, and there is a sequence $\seq{\mu_j} \subset \R$ such that
\begin{equation} \label{3.2}
\mu_j\, I'(u_j) - \frac{J'(u_j)}{J(u_j)^2} \to 0.
\end{equation}
Since $\dualp{I'(u_j)}{u_j} = p\, I(u_j) = p$, $\dualp{J'(u_j)}{u_j} = q\, J(u_j)$, and $J(u_j) \to 1/c$, then $\mu_j \to qc/p > 0$. By Lemma \ref{Lemma 2.2},
\begin{equation} \label{3.3}
\abs{\dualp{J'(u_j)}{(u_j - u)}} \le \int_\Omega |h(x)|\, |u_j|^{q-1}\, |u_j - u|\, dx \le C \norm{u_j}^{q-1} \pnorm[b]{u_j - u},
\end{equation}
where $b < p^\ast$. Since $\seq{u_j}$ is bounded in $W^{1,\, p}_0(\Omega)$, a renamed subsequence converges to some $u$ weakly in $W^{1,\, p}_0(\Omega)$ and strongly in $L^b(\Omega)$. Then
\[
\int_\Omega |\nabla u_j|^{p-2}\, \nabla u_j \cdot \nabla (u_j - u)\, dx = \dualp{I'(u_j)}{(u_j - u)} \to 0
\]
by \eqref{3.2} and \eqref{3.3}, and hence $u_j \to u$ in $W^{1,\, p}_0(\Omega)$ by the (S)$_+$ property of the $p$-Laplacian. By continuity, $J(u) = 1/c > 0$ and hence $u \in \M$.
\end{proof}

We can now define an increasing and unbounded sequence of critical values of $\Psi$ via a minimax scheme. Although the standard scheme is based on the Krasnosel$'$ski\u\i's genus, here we use a cohomological index as in Perera \cite{MR1998432}. This gives additional topological information about the associated critical points that is often useful in applications.

Let us recall the definition of the $\Z_2$-cohomological index of Fadell and Rabinowitz \cite{MR57:17677}. Let $W$ be a Banach space. For a symmetric subset $M$ of $W \setminus \set{0}$, let $\overline{M} = M/\Z_2$ be the quotient space of $M$ with each $u$ and $-u$ identified, let $f : \overline{M} \to \RP^\infty$ be the classifying map of $\overline{M}$, and let $f^\ast : H^\ast(\RP^\infty) \to H^\ast(\overline{M})$ be the induced homomorphism of the Alexander-Spanier cohomology rings. Then the cohomological index of $M$ is defined by
\[
i(M) = \begin{cases}
\sup\, \bgset{m \ge 1 : f^\ast(\omega^{m-1}) \ne 0}, & M \ne \emptyset\\[5pt]
0, & M = \emptyset,
\end{cases}
\]
where $\omega \in H^1(\RP^\infty)$ is the generator of the polynomial ring $H^\ast(\RP^\infty) = \Z_2[\omega]$. For example, the classifying map of the unit sphere $S^{m-1}$ in $\R^m$, $m \ge 1$, is the inclusion $\RP^{m-1} \incl \RP^\infty$, which induces isomorphisms on $H^q$ for $q \le m-1$, so $i(S^{m-1}) = m$.

Let $\F$ denote the class of symmetric subsets of $\M$, and set
\[
\lambda_k := \inf_{\substack{M \in \F\\[1pt]
i(M) \ge k}}\, \sup_{u \in M}\, \Psi(u), \quad k \ge 1.
\]
Then $\seq{\lambda_k}$ is a sequence of positive eigenvalues of \eqref{3.1}, $\lambda_k \nearrow + \infty$, and
\begin{equation} \label{3.4}
i(\bgset{u \in \M : \Psi(u) \le \lambda_k}) = i(\bgset{u \in \M : \Psi(u) < \lambda_{k+1}}) = k
\end{equation}
if $\lambda_k < \lambda_{k+1}$ (see Perera, Agarwal, and O'Regan \cite[Propositions 3.52 and 3.53]{MR2640827}).

As an application consider the pure power problem
\begin{equation} \label{3.5}
\left\{\begin{aligned}
- \Delta_p\, u & = h(x)\, |u|^{q-2}\, u && \text{in } \Omega\\[10pt]
u & = 0 && \text{on } \bdry{\Omega}.
\end{aligned}\right.
\end{equation}

\begin{theorem}
If $q \in [1,p^\ast)$, $q \ne p$, and $h \in \A_q$ is positive on a set of positive measure, then problem \eqref{3.5} has a sequence of nontrivial weak solutions $\seq{u_k}$ such that
\begin{enumarab}
\item if $q < p$, then $\norm{u_k} \to 0$;
\item if $q > p$, then $\norm{u_k} \to \infty$.
\end{enumarab}
\end{theorem}

\begin{proof}
Let $v_k$ be a critical point of $\Psi$ with $\Psi(v_k) = \lambda_k$. Then $u_k := \lambda_k^{1/(q-p)}\, v_k$ solves \eqref{3.5}, and $\norm{u_k} = \lambda_k^{1/(q-p)}\, p^{1/p}$ since $I(v_k) = 1$.
\end{proof}

In the next section we will use the index information in \eqref{3.4} to compute certain critical groups when $q = p$.

\section{Critical groups} \label{Section 4}

In this section we consider the problem
\begin{equation} \label{4.1}
\left\{\begin{aligned}
- \Delta_p\, u & = \lambda\, h(x)\, |u|^{p-2}\, u + g(x,u) && \text{in } \Omega\\[10pt]
u & = 0 && \text{on } \bdry{\Omega},
\end{aligned}\right.
\end{equation}
where $\lambda \ge 0$ is a parameter, $h \in \A_p$ is positive on a set of positive measure, and $g$ is a Carath\'{e}odory function on $\Omega \times \R$ satisfying the growth condition
\begin{equation} \label{4.2}
|g(x,t)| \le \sum_{i=1}^n K_i(x)\, |t|^{q_i - 1} \quad \text{for a.a. $x \in \Omega$ and all $t \in \R$}
\end{equation}
for some $q_i \in (p,p^\ast)$ and $K_i \in \A_{q_i}$. Problem \eqref{4.1} has the trivial solution $u = 0$, and we study the critical groups of the associated functional
\[
\Phi(u) = \int_\Omega \left[\frac{1}{p}\, |\nabla u|^p - \frac{1}{p}\, \lambda\, h(x)\, |u|^p - G(x,u)\right] dx, \quad u \in W^{1,\, p}_0(\Omega),
\]
where $G(x,t) = \int_0^t g(x,\tau)\, d\tau$, at $0$.

Let us recall that the critical groups of $\Phi$ at $0$ are given by
\begin{equation} \label{4.3}
C^q(\Phi,0) = H^q(\Phi^0 \cap U,\Phi^0 \cap U \setminus \set{0}), \quad q \ge 0,
\end{equation}
where $\Phi^0 = \bgset{u \in W^{1,\, p}_0(\Omega) : \Phi(u) \le 0}$, $U$ is any neighborhood of $0$, and $H$ denotes Alexander-Spanier cohomology with $\Z_2$-coefficients. They are independent of $U$ by the excision property of the cohomology groups. They are also invariant under homotopies that preserve the isolatedness of the critical point by the following proposition (see Chang and Ghoussoub \cite{MR1422006} or Corvellec and Hantoute \cite{MR1926378}).

\begin{proposition} \label{Proposition 4.1}
Let $\Phi_s,\, s \in [0,1]$ be a family of $C^1$-functionals on a Banach space $W$ such that $0$ is a critical point of each $\Phi_s$. If there is a closed neighborhood $U$ of $0$ such that
\begin{enumarab}
\item each $\Phi_s$ satisfies the {\em \PS{}} condition over $U$,
\item $U$ contains no other critical point of any $\Phi_s$,
\item the map $[0,1] \to C^1(U,\R),\, s \mapsto \Phi_s$ is continuous,
\end{enumarab}
then $C_q(\Phi_0,0) \isom C_q(\Phi_1,0)$ for all $q$.
\end{proposition}

In the absence of a direct sum decomposition, the main technical tool we use to get an estimate of the critical groups is the notion of a cohomological local splitting introduced in Perera, Agarwal, and O'Regan \cite{MR2640827}, which is a variant of the homological local linking of Perera \cite{MR1700283} (see also Li and Willem \cite{MR1312028}). The following slightly different form of this notion was given in Degiovanni, Lancelotti, and Perera \cite{MR2661274}.

\begin{definition}
We say that a $C^1$-functional $\Phi$ on a Banach space $W$ has a cohomological local splitting near $0$ in dimension $k \ge 1$ if there are symmetric cones $W_\pm \subset W$ with $W_+ \cap W_- = \set{0}$ and $\rho > 0$ such that
\begin{equation} \label{4.4}
i(W_- \setminus \set{0}) = i(W \setminus W_+) = k
\end{equation}
and
\begin{equation} \label{4.5}
\Phi(u) \le \Phi(0) \quad \forall u \in B_\rho \cap W_-, \qquad \Phi(u) \ge \Phi(0) \quad \forall u \in B_\rho \cap W_+,
\end{equation}
where $B_\rho = \bgset{u \in W : \norm{u} \le \rho}$.
\end{definition}

\begin{proposition}[{Degiovanni, Lancelotti, and Perera \cite[Proposition 2.1]{MR2661274}}] \label{Proposition 4.2}
If $\Phi$ has a cohomological local splitting near $0$ in dimension $k$, and $0$ is an isolated critical point of $\Phi$, then $C^k(\Phi,0) \ne 0$.
\end{proposition}

Let $\lambda_k \nearrow + \infty$ be the sequence of positive eigenvalues of the problem
\begin{equation} \label{4.6}
\left\{\begin{aligned}
- \Delta_p\, u & = \lambda\, h(x)\, |u|^{p-2}\, u && \text{in } \Omega\\[10pt]
u & = 0 && \text{on } \bdry{\Omega}
\end{aligned}\right.
\end{equation}
that was constructed in the last section. The main result of this section is the following theorem.

\begin{theorem} \label{Theorem 4.4}
Assume that $h \in \A_p$ is positive on a set of positive measure, $g$ satisfies \eqref{4.2} for some $q_i \in (p,p^\ast)$ and $K_i \in \A_{q_i}$ for $i = 1,\dots,n$, and $0$ is an isolated critical point of $\Phi$.
\begin{enumarab}
\item $C^0(\Phi,0) \isom \Z_2$ and $C^q(\Phi,0) = 0$ for $q \ge 1$ in the following cases:
\begin{enumroman}
\item $0 \le \lambda < \lambda_1$;
\item $\lambda = \lambda_1$ and $G(x,t) \le 0$ for a.a. $x \in \Omega$ and all $t \in \R$.
\end{enumroman}
\item $C^k(\Phi,0) \ne 0$ in the following cases:
\begin{enumroman}
\item $\lambda_k < \lambda < \lambda_{k+1}$;
\item $\lambda = \lambda_k < \lambda_{k+1}$ and $G(x,t) \ge 0$ for a.a. $x \in \Omega$ and all $t \in \R$;
\item $\lambda_k < \lambda_{k+1} = \lambda$ and $G(x,t) \le 0$ for a.a. $x \in \Omega$ and all $t \in \R$.
\end{enumroman}
\end{enumarab}
\end{theorem}

\begin{proof}
By \eqref{4.2} and Lemma \ref{Lemma 2.2},
\[
\abs{\int_\Omega G(x,u)\, dx} \le C \sum_{i=1}^n \norm{u}^{q_i} = \o(\norm{u}^p) \quad \text{as } \norm{u} \to 0
\]
since each $q_i > p$. So
\begin{equation} \label{4.7}
\Phi(u) = I(u) - \lambda\, J(u) + \o(\norm{u}^p) \quad \text{as } \norm{u} \to 0.
\end{equation}

(1) We show that $0$ is a local minimizer of $\Phi$. Since $\Psi(u) \ge \lambda_1$ for all $u \in \M$,
\begin{equation} \label{4.8}
I(u) \ge \lambda_1\, J(u) \quad \forall u \in W^{1,\, p}_0(\Omega).
\end{equation}

({\em i}) For sufficiently small $\rho > 0$,
\[
\Phi(u) \ge \left(1 - \frac{\lambda}{\lambda_1} + \o(1)\right) \frac{\norm{u}^p}{p} \ge 0 \quad \forall u \in B_\rho
\]
by \eqref{4.7} and \eqref{4.8}.

({\em ii}) We have
\[
\Phi(u) \ge - \int_\Omega G(x,u)\, dx \ge 0 \quad \forall u \in W^{1,\, p}_0(\Omega).
\]

(2) We show that $\Phi$ has a cohomological local splitting near $0$ in dimension $k$ and apply Proposition \ref{Proposition 4.2}. Let
\[
W_- = \bgset{u \in W^{1,\, p}_0(\Omega) : I(u) \le \lambda_k\, J(u)}, \qquad W_+ = \bgset{u \in W^{1,\, p}_0(\Omega) : I(u) \ge \lambda_{k+1}\, J(u)}.
\]
Then $W_- \setminus \set{0}$ and $W \setminus W_+$ radially deformation retract to $\bgset{u \in \M : \Psi(u) \le \lambda_k}$ and $\bgset{u \in \M : \Psi(u) < \lambda_{k+1}}$, respectively, so \eqref{4.4} holds by \eqref{3.4}. It only remains to show that \eqref{4.5} holds for sufficiently small $\rho > 0$.

({\em i}) For sufficiently small $\rho > 0$,
\[
\Phi(u) \le - \left(\frac{\lambda}{\lambda_k} - 1 + \o(1)\right) \frac{\norm{u}^p}{p} \le 0 \quad \forall u \in B_\rho \cap W_-
\]
and
\[
\Phi(u) \ge \left(1 - \frac{\lambda}{\lambda_{k+1}} + \o(1)\right) \frac{\norm{u}^p}{p} \ge 0 \quad \forall u \in B_\rho \cap W_+
\]
by \eqref{4.7}.

({\em ii}) We have
\[
\Phi(u) \le - \int_\Omega G(x,u)\, dx \le 0 \quad \forall u \in W_-,
\]
and for sufficiently small $\rho > 0$, $\Phi(u) \ge 0$ for all $u \in B_\rho \cap W_+$ as in ({\em i}).

({\em iii}) For sufficiently small $\rho > 0$, $\Phi(u) \le 0$ for all $u \in B_\rho \cap W_-$ as in ({\em i}), and
\[
\Phi(u) \ge - \int_\Omega G(x,u)\, dx \ge 0 \quad \forall u \in W_+. \QED
\]
\end{proof}

When $p > N$, it suffices to assume the sign conditions on $G$ in Theorem \ref{Theorem 4.4} for small $|t|$ by the imbedding $W^{1,\, p}_0(\Omega) \hookrightarrow L^\infty(\Omega)$, so we also have the following theorem.

\begin{theorem} \label{Theorem 4.5}
Assume that $p > N$, $h \in \A_p$ is positive on a set of positive measure, $g$ satisfies \eqref{4.2} for some $q_i \in (p,\infty)$ and $K_i \in \A_{q_i}$ for $i = 1,\dots,n$, and $0$ is an isolated critical point of $\Phi$.
\begin{enumarab}
\item $C^0(\Phi,0) \isom \Z_2$ and $C^q(\Phi,0) = 0$ for $q \ge 1$ if $\lambda = \lambda_1$ and, for some $\delta > 0$, $G(x,t) \le 0$ for a.a. $x \in \Omega$ and $|t| \le \delta$.
\item $C^k(\Phi,0) \ne 0$ in the following cases:
\begin{enumroman}
\item $\lambda = \lambda_k < \lambda_{k+1}$ and, for some $\delta > 0$, $G(x,t) \ge 0$ for a.a. $x \in \Omega$ and $|t| \le \delta$;
\item $\lambda_k < \lambda_{k+1} = \lambda$ and, for some $\delta > 0$, $G(x,t) \le 0$ for a.a. $x \in \Omega$ and $|t| \le \delta$.
\end{enumroman}
\end{enumarab}
\end{theorem}

We close this section by showing that the conclusions of Theorem \ref{Theorem 4.5} also hold for $p \le N$ when the weights $h$ and $K_i$ belong to suitable subclasses of $\A_p$ and $\A_{q_i}$, respectively.

\begin{definition} \label{Definition 4.6}
For $p \le N$ and $q \in [1,p^\ast)$, let $\widetilde{\A}_q$ denote the class of measurable functions $K$ such that $K \rho^a \in L^r(\Omega)$ for some $a \in [0,q - 1]$ and $r \in (1,\infty)$ satisfying
\begin{equation} \label{4.9}
\frac{1}{r} + \frac{a}{p} + \frac{q-1-a}{p^\ast} < \frac{p}{N}.
\end{equation}
\end{definition}

Note that $\widetilde{\A}_q = \A_q$ when $p = N$. When $p < N$,
\[
\frac{1}{p^\ast} + \frac{p}{N} = 1 - \frac{(N-p)(p-1)}{Np} < 1
\]
and hence $\widetilde{\A}_q \subset \A_q$.

\begin{lemma} \label{Lemma 4.7}
If $p \le N$, $q \in [1,p^\ast)$, and $K \in \widetilde{\A}_q$, then there exists $s > N/p$ such that $K(x)\, |u|^{q-1} \in L^s(\Omega)$ and
\[
\pnorm[s]{K(x)\, |u|^{q-1}} \le C \norm{u}^{q-1}
\]
for all $u \in W^{1,\, p}_0(\Omega)$.
\end{lemma}

\begin{proof}
Let $a$ and $r$ be as in Definition \ref{Definition 4.6}. By \eqref{4.9}, there exists $b < p^\ast$ such that
\begin{equation} \label{4.10}
\frac{1}{r} + \frac{a}{p} + \frac{q-1-a}{b} < \frac{p}{N}.
\end{equation}
By the H\"older inequality,
\[
\int_\Omega |K(x)|^s\, |u|^{(q-1)\, s}\, dx = \int_\Omega \abs{K \rho^a}^s \abs{\frac{u}{\rho}}^{as} |u|^{(q-1-a)\, s}\, dx \le \pnorm[r]{K \rho^a}^s \pnorm[p]{\frac{u}{\rho}}^{as} \pnorm[b]{u}^{(q-1-a)\, s},
\]
where $s/r + as/p + (q-1-a)\, s/b = 1$ and hence $s > N/p$ by \eqref{4.10}. Since $\pnorm[p]{u/\rho} \le C \norm{u}$ by the Hardy inequality (see Ne{\v{c}}as \cite{MR0163054}) and $\pnorm[b]{u} \le C \norm{u}$ by the Sobolev imbedding, the conclusion follows.
\end{proof}

Assume that $p \le N$, $h \in \widetilde{\A}_p$, and $K_i \in \widetilde{\A}_{q_i}$ for $i = 1,\dots,n$. First we show that the critical groups of $\Phi$ at $0$ depend only on the values of $g(x,t)$ for small $|t|$.

\begin{lemma} \label{Lemma 4.8}
Let $\delta > 0$ and let $\vartheta : \R \to [- \delta,\delta]$ be a smooth nondecreasing function such that $\vartheta(t) = - \delta$ for $t \le - \delta$, $\vartheta(t) = t$ for $- \delta/2 \le t \le \delta/2$, and $\vartheta(t) = \delta$ for $t \ge \delta$. Set
\[
\Phi_1(u) = \int_\Omega \left[\frac{1}{p}\, |\nabla u|^p - \frac{1}{p}\, \lambda\, h(x)\, |u|^p - G(x,\vartheta(u))\right] dx, \quad u \in W^{1,\, p}_0(\Omega).
\]
If $0$ is an isolated critical point of $\Phi$, then it is also an isolated critical point of $\Phi_1$ and $C_q(\Phi,0) \isom C_q(\Phi_1,0)$ for all $q$.
\end{lemma}

\begin{proof}
We apply Proposition \ref{Proposition 4.1} to the family of functionals
\[
\Phi_s(u) = \int_\Omega \left[\frac{1}{p}\, |\nabla u|^p - \frac{1}{p}\, \lambda\, h(x)\, |u|^p - G(x,(1 - s)\, u + s\, \vartheta(u))\right] dx, \quad u \in W^{1,\, p}_0(\Omega),\, s \in [0,1]
\]
in a small ball $B_\varepsilon(0) = \bgset{u \in W^{1,\, p}_0(\Omega) : \norm{u} \le \varepsilon}$, noting that $\Phi_0 = \Phi$. Lemma \ref{Lemma 2.3} implies that each $\Phi_s$ satisfies the \PS{} condition over $B_\varepsilon(0)$, and it is easy to see that the map $[0,1] \to C^1(B_\varepsilon(0),\R),\, s \mapsto \Phi_s$ is continuous, so it only remains to show that for sufficiently small $\varepsilon > 0$, $B_\varepsilon(0)$ contains no critical point of any $\Phi_s$ other than $0$.

Suppose $u_j \to 0$ in $W^{1,\, p}_0(\Omega)$, $\Phi_{s_j}'(u_j) = 0,\, s_j \in [0,1]$, and $u_j \ne 0$. Then
\[
\left\{\begin{aligned}
- \Delta_p\, u_j & = \lambda\, h(x)\, |u_j|^{p-2}\, u_j + g_j(x,u_j) && \text{in } \Omega\\[10pt]
u_j & = 0 && \text{on } \bdry{\Omega},
\end{aligned}\right.
\]
where
\[
g_j(x,t) = (1 - s_j + s_j\, \vartheta'(t))\, g(x,(1 - s_j)\, t + s_j\, \vartheta(t)).
\]
Since $(1 - s_j)\, t + s_j\, \vartheta(t) = t$ for $|t| \le \delta/2$ and $|(1 - s_j)\, t + s_j\, \vartheta(t)| \le |t| + \delta < 3\, |t|$ for $|t| > \delta/2$, \eqref{4.2} implies
\begin{equation} \label{4.11}
|g_j(x,t)| \le C \sum_{i=1}^n K_i(x)\, |t|^{q_i - 1} \quad \text{for a.a. $x \in \Omega$ and all $t \in \R$},
\end{equation}
where $C$ denotes a generic positive constant independent of $j$. By Lemma \ref{Lemma 4.7}, there exists $s > N/p$ such that $h(x)\, |u_j|^{p-1} \in L^s(\Omega)$ and
\begin{equation}
\pnorm[s]{h(x)\, |u_j|^{p-1}} \le C \norm{u_j}^{p-1},
\end{equation}
and there exists $s_i > N/p$ such that $K_i(x)\, |u_j|^{q_i - 1} \in L^{s_i}(\Omega)$ and
\begin{equation} \label{4.13}
\pnorm[s_i]{K_i(x)\, |u_j|^{q_i - 1}} \le C \norm{u_j}^{q_i - 1}
\end{equation}
for $i = 1,\dots,n$. Let $s_0 = \min\, \bgset{s,s_1,\dots,s_n} > N/p$. By \eqref{4.11}--\eqref{4.13}, $\lambda\, h(x)\, |u_j|^{p-2}\, u_j + g_j(x,u_j) \in L^{s_0}(\Omega)$ and
\begin{multline*}
\pnorm[s_0]{\lambda\, h(x)\, |u_j|^{p-2}\, u_j + g_j(x,u_j)} \le C \left(\pnorm[s_0]{h(x)\, |u_j|^{p-1}} + \sum_{i=1}^n \pnorm[s_0]{K_i(x)\, |u_j|^{q_i - 1}}\right)\\[10pt]
\le C \left(\pnorm[s]{h(x)\, |u_j|^{p-1}} + \sum_{i=1}^n \pnorm[s_i]{K_i(x)\, |u_j|^{q_i - 1}}\right) \le C \left(\norm{u_j}^{p-1} + \sum_{i=1}^n \norm{u_j}^{q_i - 1}\right) \to 0,
\end{multline*}
and hence $u_j \in L^\infty(\Omega)$ and $u_j \to 0$ in $L^\infty(\Omega)$ by Guedda and V{\'e}ron \cite[Proposition 1.3]{MR1009077}. So for sufficiently large $j$, $|u_j| \le \delta/2$ a.e.\! and hence $\Phi'(u_j) = \Phi_{s_j}'(u_j) = 0$, contradicting our assumption that $0$ is an isolated critical point of $\Phi$.
\end{proof}

The following theorem is now immediate from Lemma \ref{Lemma 4.8} and Theorem \ref{Theorem 4.4}.

\begin{theorem} \label{Theorem 4.9}
Assume that $p \le N$, $h \in \widetilde{\A}_p$ is positive on a set of positive measure, $g$ satisfies \eqref{4.2} for some $q_i \in (p,p^\ast)$ and $K_i \in \widetilde{\A}_{q_i}$ for $i = 1,\dots,n$, and $0$ is an isolated critical point of $\Phi$.
\begin{enumarab}
\item $C^0(\Phi,0) \isom \Z_2$ and $C^q(\Phi,0) = 0$ for $q \ge 1$ if $\lambda = \lambda_1$ and, for some $\delta > 0$, $G(x,t) \le 0$ for a.a. $x \in \Omega$ and $|t| \le \delta$.
\item $C^k(\Phi,0) \ne 0$ in the following cases:
\begin{enumroman}
\item $\lambda = \lambda_k < \lambda_{k+1}$ and, for some $\delta > 0$, $G(x,t) \ge 0$ for a.a. $x \in \Omega$ and $|t| \le \delta$;
\item $\lambda_k < \lambda_{k+1} = \lambda$ and, for some $\delta > 0$, $G(x,t) \le 0$ for a.a. $x \in \Omega$ and $|t| \le \delta$.
\end{enumroman}
\end{enumarab}
\end{theorem}

\section{Nontrivial solutions} \label{Section 5}

In this section we obtain a nontrivial solution of the problem
\begin{equation} \label{5.1}
\left\{\begin{aligned}
- \Delta_p\, u & = \lambda\, h(x)\, |u|^{p-2}\, u + K(x)\, |u|^{q-2}\, u + g(x,u) && \text{in } \Omega\\[10pt]
u & = 0 && \text{on } \bdry{\Omega},
\end{aligned}\right.
\end{equation}
where $q \in (p,p^\ast)$, $K \in \A_q$ satisfies
\begin{equation} \label{5.2}
\essinf_{x \in \Omega}\, K(x) > 0,
\end{equation}
and $g$ satisfies \eqref{4.2} with each $q_i \in (p,q)$. We will assume that the weights $h$ and $K_i$ belong to suitable subclasses of $\A_p$ and $\A_{q_i}$, respectively.

\begin{definition} \label{Definition 5.1}
For $q \in (1,p^\ast)$ and $s \in [1,q)$, let $\A_s^q$ denote the class of measurable functions $K$ such that $K \rho^a \in L^r(\Omega)$ for some $a \in [0,s - 1]$ and $r \in (1,\infty)$ satisfying
\begin{equation} \label{5.3}
\frac{1}{r} + \frac{a}{p} + \frac{s-a}{q} \le 1.
\end{equation}
\end{definition}

Clearly, $\A_s^q \subset \A_s$.

\begin{lemma} \label{Lemma 5.2}
If $q \in [1,p^\ast)$, $s \in [1,q)$, and $K \in \A_s^q$, then there exist $t < p$ and, for every $\varepsilon > 0$, a constant $C(\varepsilon)$ such that
\[
\int_\Omega |K(x)|\, |u|^s\, dx \le C(\varepsilon) \norm{u}^t + \varepsilon \pnorm[q]{u}^q \quad \forall u \in W^{1,\, p}_0(\Omega).
\]
\end{lemma}

\begin{proof}
Let $a$ and $r$ be as in Definition \ref{Definition 5.1}. By the H\"older inequality,
\[
\int_\Omega |K(x)|\, |u|^s\, dx = \int_\Omega \abs{K \rho^a} \abs{\frac{u}{\rho}}^a |u|^{s-a}\, dx \le \pnorm[r]{K \rho^a} \pnorm[p]{\frac{u}{\rho}}^a \pnorm[b]{u}^{s-a},
\]
where $1/r + a/p + (s-a)/b = 1$ and hence $b \le q$ by \eqref{5.3}. Since $\pnorm[p]{u/\rho} \le C \norm{u}$ by the Hardy inequality (see Ne{\v{c}}as \cite{MR0163054}) and $\pnorm[b]{u} \le C \pnorm[q]{u}$, the last expression is less than or equal to $C \norm{u}^{a} \pnorm[q]{u}^{s - a}$. By the Young inequality, the latter is less than or equal to $C(\varepsilon) \norm{u}^{t} + \varepsilon \pnorm[q]{u}^q$, where $a/t + (s-a)/q = 1$ and hence $t < p$ by \eqref{5.3}.
\end{proof}

We assume that $h \in \A_p^q$ and $K_i \in \A_{q_i}^q$ for $i = 1,\dots,n$. First we verify that the associated functional
\[
\Phi(u) = \int_\Omega \left[\frac{1}{p}\, |\nabla u|^p - \frac{1}{p}\, \lambda\, h(x)\, |u|^p - \frac{1}{q}\, K(x)\, |u|^q - G(x,u)\right] dx, \quad u \in W^{1,\, p}_0(\Omega),
\]
where $G(x,t) = \int_0^t g(x,\tau)\, d\tau$, satisfies the \PS{} condition.

\begin{lemma} \label{Lemma 5.3}
Every sequence $\seq{u_j} \subset W^{1,\, p}_0(\Omega)$ such that $\seq{\Phi(u_j)}$ is bounded and $\Phi'(u_j) \to 0$ has a convergent subsequence.
\end{lemma}

\begin{proof}
It suffices to show that $\seq{u_j}$ is bounded by Lemma \ref{Lemma 2.3}. We have
\begin{equation} \label{5.4}
\left(1 - \frac{p}{q}\right) \int_\Omega K(x)\, |u_j|^q\, dx = p\, \Phi(u_j) - \dualp{\Phi'(u_j)}{u_j} + \int_\Omega \left[p\, G(x,u_j) - u_j\, g(x,u_j)\right] dx.
\end{equation}
By \eqref{4.2} and Lemma \ref{Lemma 5.2}, for every $\varepsilon > 0$, there exists $C(\varepsilon)$ such that
\begin{equation} \label{5.5}
\abs{\int_\Omega \left[p\, G(x,u_j) - u_j\, g(x,u_j)\right] dx} \le \sum_{i=1}^n \left(1 + \frac{p}{q_i}\right) \int_\Omega K_i(x)\, |u_j|^{q_i}\, dx \le C(\varepsilon) \sum_{i=1}^n \norm{u_j}^{t_i} + \varepsilon \pnorm[q]{u_j}^q,
\end{equation}
where each $t_i < p$. Combining \eqref{5.2}, \eqref{5.4}, and \eqref{5.5} gives
\begin{equation} \label{5.6}
\pnorm[q]{u_j}^q \le C \left(\sum_{i=1}^n \norm{u_j}^{t_i} + 1\right) + \o(\norm{u_j}).
\end{equation}

Now we use
\begin{equation}
\left(\frac{q}{p} - 1\right) \left(\norm{u_j}^p - \lambda \int_\Omega h(x)\, |u_j|^p\, dx\right) = q\, \Phi(u_j) - \dualp{\Phi'(u_j)}{u_j} + \int_\Omega \left[q\, G(x,u_j) - u_j\, g(x,u_j)\right] dx.
\end{equation}
By Lemma \ref{Lemma 5.2},
\begin{equation}
\abs{\int_\Omega h(x)\, |u_j|^p\, dx} \le C \left(\norm{u_j}^t + \pnorm[q]{u_j}^q\right),
\end{equation}
where $t < p$. As in \eqref{5.5},
\begin{equation} \label{5.9}
\abs{\int_\Omega \left[q\, G(x,u_j) - u_j\, g(x,u_j)\right] dx} \le C \left(\sum_{i=1}^n \norm{u_j}^{t_i} + \pnorm[q]{u_j}^q\right).
\end{equation}
Combining \eqref{5.6}--\eqref{5.9} gives
\[
\norm{u_j}^p \le C \left(\sum_{i=1}^n \norm{u_j}^{t_i} + \norm{u_j}^t + 1\right) + \o(\norm{u_j}),
\]
which implies that $\seq{u_j}$ is bounded since each $t_i < p$ and $t < p$.
\end{proof}

Next we study the structure of the sublevel sets $\Phi^\alpha = \bgset{u \in W^{1,\, p}_0(\Omega) : \Phi(u) \le \alpha}$ for $\alpha < 0$ with $|\alpha|$ large.

\begin{lemma} \label{Lemma 5.4}
We have
\begin{enumarab}
\item $\displaystyle \sup_{u \in W^{1,\, p}_0(\Omega)} \left(\dualp{\Phi'(u)}{u} - \frac{p+q}{2}\, \Phi(u)\right) < + \infty$;
\item $\displaystyle \lim_{t \to + \infty} \Phi(tu) = - \infty \quad \forall u \in W^{1,\, p}_0(\Omega) \setminus \set{0}$.
\end{enumarab}
\end{lemma}

\begin{proof}
(1) We have
\begin{multline} \label{5.10}
\dualp{\Phi'(u)}{u} - \frac{p+q}{2}\, \Phi(u) = - \frac{q-p}{2} \int_\Omega \left[\frac{1}{p}\, |\nabla u|^p - \frac{1}{p}\, \lambda\, h(x)\, |u|^p + \frac{1}{q}\, K(x)\, |u|^q\right] dx\\[10pt]
+ \int_\Omega \left[\frac{p+q}{2}\, G(x,u) - u\, g(x,u)\right] dx.
\end{multline}
By \eqref{4.2} and Lemma \ref{Lemma 5.2}, for every $\varepsilon > 0$, there exists $C(\varepsilon)$ such that
\begin{gather}
\abs{\int_\Omega \left[\frac{p+q}{2}\, G(x,u) - u\, g(x,u)\right] dx} \le C(\varepsilon) \sum_{i=1}^n \norm{u}^{t_i} + \varepsilon \pnorm[q]{u}^q,\\[10pt]
\label{5.12} \abs{\int_\Omega h(x)\, |u|^p\, dx} \le C(\varepsilon) \norm{u}^t + \varepsilon \pnorm[q]{u}^q,
\end{gather}
where each $t_i < p$ and $t < p$. Combining \eqref{5.2} and \eqref{5.10}--\eqref{5.12} gives
\[
\dualp{\Phi'(u)}{u} - \frac{p+q}{2}\, \Phi(u) \le - \frac{1}{2} \left(\frac{q}{p} - 1\right) \norm{u}^p + C \left(\sum_{i=1}^n \norm{u}^{t_i} + \norm{u}^t\right),
\]
from which the conclusion follows.

(2) This follows from \eqref{5.2} and \eqref{4.2} since $p < q$ and each $q_i < q$.
\end{proof}

\begin{lemma} \label{Lemma 5.5}
There exists $\alpha < 0$ such that $\Phi^\alpha$ is contractible in itself.
\end{lemma}

\begin{proof}
By Lemma \ref{Lemma 5.4} (1), there exists $\alpha < 0$ such that
\begin{equation} \label{5.13}
\dualp{\Phi'(u)}{u} < 0 \quad \forall u \in \Phi^\alpha.
\end{equation}
For $u \in W^{1,\, p}_0(\Omega) \setminus \set{0}$, taking into account Lemma \ref{Lemma 5.4} (2), set
\[
t(u) = \min\, \bgset{t \ge 1 : \Phi(tu) \le \alpha},
\]
and note that the function $u \mapsto t(u)$ is continuous by \eqref{5.13}. Then $u \mapsto t(u)\, u$ is a retraction of $W^{1,\, p}_0(\Omega) \setminus \set{0}$ onto $\Phi^\alpha$, and the conclusion follows since $W^{1,\, p}_0(\Omega) \setminus \set{0}$ is contractible in itself.
\end{proof}

We are now ready to prove our main existence result. Let $\lambda_k \nearrow + \infty$ be the sequence of positive eigenvalues of problem \eqref{4.6} considered in the last section.

\begin{theorem} \label{Theorem 5.6}
Assume that $\lambda \ge 0$, $q \in (p,p^\ast)$, $K \in \A_q$ satisfies \eqref{5.2}, $h \in \A_p^q$ is positive on a set of positive measure, and $g$ satisfies \eqref{4.2} for some $q_i \in (p,q)$ and $K_i \in \A_{q_i}^q$ for $i = 1,\dots,n$. Then problem \eqref{5.1} has a nontrivial weak solution in each of the following cases:
\begin{enumarab}
\item $\lambda \notin \bgset{\lambda_k : k \ge 1}$;
\item $G(x,t) \ge 0$ for a.a. $x \in \Omega$ and all $t \in \R$;
\item $G(x,t) \le 0$ for a.a. $x \in \Omega$ and all $t \in \R$;
\item $p > N$ and, for some $\delta > 0$, $G(x,t) \ge 0$ for a.a. $x \in \Omega$ and $|t| \le \delta$;
\item $p > N$ and, for some $\delta > 0$, $G(x,t) \le 0$ for a.a. $x \in \Omega$ and $|t| \le \delta$;
\item $p \le N$, $h \in \widetilde{\A}_p$, $K_i \in \widetilde{\A}_{q_i}$ for $i = 1,\dots,n$, and, for some $\delta > 0$, $G(x,t) \ge 0$ for a.a. $x \in \Omega$ and $|t| \le \delta$;
\item $p \le N$, $h \in \widetilde{\A}_p$, $K_i \in \widetilde{\A}_{q_i}$ for $i = 1,\dots,n$, and, for some $\delta > 0$, $G(x,t) \le 0$ for a.a. $x \in \Omega$ and $|t| \le \delta$.
\end{enumarab}
\end{theorem}

\begin{proof}
Suppose that $0$ is the only critical point of $\Phi$. Taking $U = W^{1,\, p}_0(\Omega)$ in \eqref{4.3}, we have
\[
C^q(\Phi,0) = H^q(\Phi^0,\Phi^0 \setminus \set{0}).
\]
Let $\alpha < 0$ be as in Lemma \ref{Lemma 5.5}. Since $\Phi$ has no other critical points and satisfies the \PS{} condition by Lemma \ref{Lemma 5.3}, $\Phi^0$ is a deformation retract of $W^{1,\, p}_0(\Omega)$ and $\Phi^\alpha$ is a deformation retract of $\Phi^0 \setminus \set{0}$ by the second deformation lemma. So
\[
C^q(\Phi,0) \isom H^q(W^{1,\, p}_0(\Omega),\Phi^\alpha) = 0 \quad \forall q
\]
since $\Phi^\alpha$ is contractible in itself, contradicting Theorem \ref{Theorem 4.4}, Theorem \ref{Theorem 4.5}, or Theorem \ref{Theorem 4.9}.
\end{proof}

\def\cdprime{$''$}


\begin{thebibliography}{10}

\bibitem{MR1422006}
K.~C. Chang and N.~Ghoussoub.
\newblock The {C}onley index and the critical groups via an extension of
  {G}romoll-{M}eyer theory.
\newblock {\em Topol. Methods Nonlinear Anal.}, 7(1):77--93, 1996.

\bibitem{MR1926378}
J.-N. Corvellec and A.~Hantoute.
\newblock Homotopical stability of isolated critical points of continuous
  functionals.
\newblock {\em Set-Valued Anal.}, 10(2-3):143--164, 2002.
\newblock Calculus of variations, nonsmooth analysis and related topics.

\bibitem{MR1836801}
Mabel Cuesta.
\newblock Eigenvalue problems for the {$p$}-{L}aplacian with indefinite
  weights.
\newblock {\em Electron. J. Differential Equations}, pages No. 33, 9, 2001.

\bibitem{MR2661274}
Marco Degiovanni, Sergio Lancelotti, and Kanishka Perera.
\newblock Nontrivial solutions of {$p$}-superlinear {$p$}-{L}aplacian problems
  via a cohomological local splitting.
\newblock {\em Commun. Contemp. Math.}, 12(3):475--486, 2010.

\bibitem{MR57:17677}
Edward~R. Fadell and Paul~H. Rabinowitz.
\newblock Generalized cohomological index theories for {L}ie group actions with
  an application to bifurcation questions for {H}amiltonian systems.
\newblock {\em Invent. Math.}, 45(2):139--174, 1978.

\bibitem{MR1009077}
Mohammed Guedda and Laurent V{\'e}ron.
\newblock Quasilinear elliptic equations involving critical {S}obolev
  exponents.
\newblock {\em Nonlinear Anal.}, 13(8):879--902, 1989.

\bibitem{MR2674189}
Ryuji Kajikiya.
\newblock Superlinear elliptic equations with singular coefficients on the
  boundary.
\newblock {\em Nonlinear Anal.}, 73(7):2117--2131, 2010.

\bibitem{MR2409516}
Ryuji Kajikiya, Yong-Hoon Lee, and Inbo Sim.
\newblock One-dimensional {$p$}-{L}aplacian with a strong singular indefinite
  weight. {I}. {E}igenvalue.
\newblock {\em J. Differential Equations}, 244(8):1985--2019, 2008.

\bibitem{MR2833356}
Ryuji Kajikiya, Yong-Hoon Lee, and Inbo Sim.
\newblock Bifurcation of sign-changing solutions for one-dimensional
  {$p$}-{L}aplacian with a strong singular weight: {$p$}-superlinear at
  {$\infty$}.
\newblock {\em Nonlinear Anal.}, 74(17):5833--5843, 2011.

\bibitem{MR1312028}
Shu~Jie Li and Michel Willem.
\newblock Applications of local linking to critical point theory.
\newblock {\em J. Math. Anal. Appl.}, 189(1):6--32, 1995.

\bibitem{MR2914567}
Marcelo Montenegro and Sebasti{\'a}n Lorca.
\newblock The spectrum of the {$p$}-{L}aplacian with singular weight.
\newblock {\em Nonlinear Anal.}, 75(9):3746--3753, 2012.

\bibitem{MR0163054}
Jind{\v{r}}ich Ne{\v{c}}as.
\newblock Sur une m\'ethode pour r\'esoudre les \'equations aux d\'eriv\'ees
  partielles du type elliptique, voisine de la variationnelle.
\newblock {\em Ann. Scuola Norm. Sup. Pisa (3)}, 16:305--326, 1962.

\bibitem{MR1700283}
Kanishka Perera.
\newblock Homological local linking.
\newblock {\em Abstr. Appl. Anal.}, 3(1-2):181--189, 1998.

\bibitem{MR1998432}
Kanishka Perera.
\newblock Nontrivial critical groups in {$p$}-{L}aplacian problems via the
  {Y}ang index.
\newblock {\em Topol. Methods Nonlinear Anal.}, 21(2):301--309, 2003.

\bibitem{MR2640827}
Kanishka Perera, Ravi~P. Agarwal, and Donal O'Regan.
\newblock {\em Morse theoretic aspects of {$p$}-{L}aplacian type operators},
  volume 161 of {\em Mathematical Surveys and Monographs}.
\newblock American Mathematical Society, Providence, RI, 2010.

\end{thebibliography}
\end{document}